\documentclass{au}
\usepackage{amssymb,latexsym,newlattice}
\usepackage{amsmath}
\usepackage[dvipdfm]{graphicx}
 \usepackage{color}
 \def\red#1{\textcolor{red}{#1}}

\def\sajat#1{}
\newcommand{\SC}[1]{\msf{C}_{#1}}
\newcommand{\SN}[1]{\msf{N}_{#1}}

\newcommand{\anchor}[2]{\textup{Anchor}_{#1}(#2)}
\newcommand{\rank}{\textup{rank}}
\DeclareMathOperator{\interior}{Inter}

\theoremstyle{plain}
\newtheorem{theorem}{Theorem}

\newtheorem{lemma}[theorem]{Lemma}
\newtheorem{corollary}[theorem]{Corollary}

\theoremstyle{definition}

\newtheorem{remark}[theorem]{Remark}

\begin{document}
\title[On the structure theorem of planar semimodular lattices]
{Notes on planar semimodular lattices. VII.\\
Resections of planar semimodular lattices}  

\author[G.\ Cz\'edli]{G\'abor Cz\'edli}
\email{czedli@math.u-szeged.hu}
\urladdr{http://www.math.u-szeged.hu/
%$\sim$czedli/}
~czedli/}

\address{University of Szeged\\Bolyai Institute\\
Szeged, Aradi v\'ertan\'uk tere 1\\HUNGARY 6720}

\author[G.\ Gr\"atzer]{George Gr\"atzer}
\email{gratzer@me.com} 
\urladdr{http://server.math.umanitoba.ca/homepages/gratzer/}
\address{Department of Mathematics\\
University of Manitoba\\
Winnipeg, MB R3T 2N2\\Canada}

\thanks{This research was supported 
by the NFSR of Hungary (OTKA), grant numbers K77432 and K83219}



\date{June 5, 2012}

\begin{abstract} 
A recent result of G.~Cz\'edli and E.\,T.~Schmidt gives 
a construction of~slim (planar) semimodular lattices  
from planar distributive lattices by adding elements, adding ``forks''. 
We give a construction that accomplishes the same by~deleting elements,
by ``resections''.
\end{abstract}

\leftline{\red{\hfill June 23, 2012.}}
\maketitle

\section{Introduction}
In this paper, we present a construction of slim (planar) 
semimodular lattices from planar distributive lattices 
by a series of \emph{resections}.
A resection starts with a cover-preserving $\SC{3}^2$ 
(the dark gray square of the three-element chain 
in Figure~\ref{fig:resect}), 
and it deletes two elements to~get an $\SN 7$ 
(see Figure~\ref{fig:Nseven}) from~$\SC{3}^2$, 
and then deletes some more elements (all the black-filled ones),  
going up and down to the left and to the right, 
to~preserve semimodularity; 
see Figure~\ref{fig:afterresect} for the result of the resection. 

\begin{figure}[p]
\centerline{\includegraphics[scale=1.0]{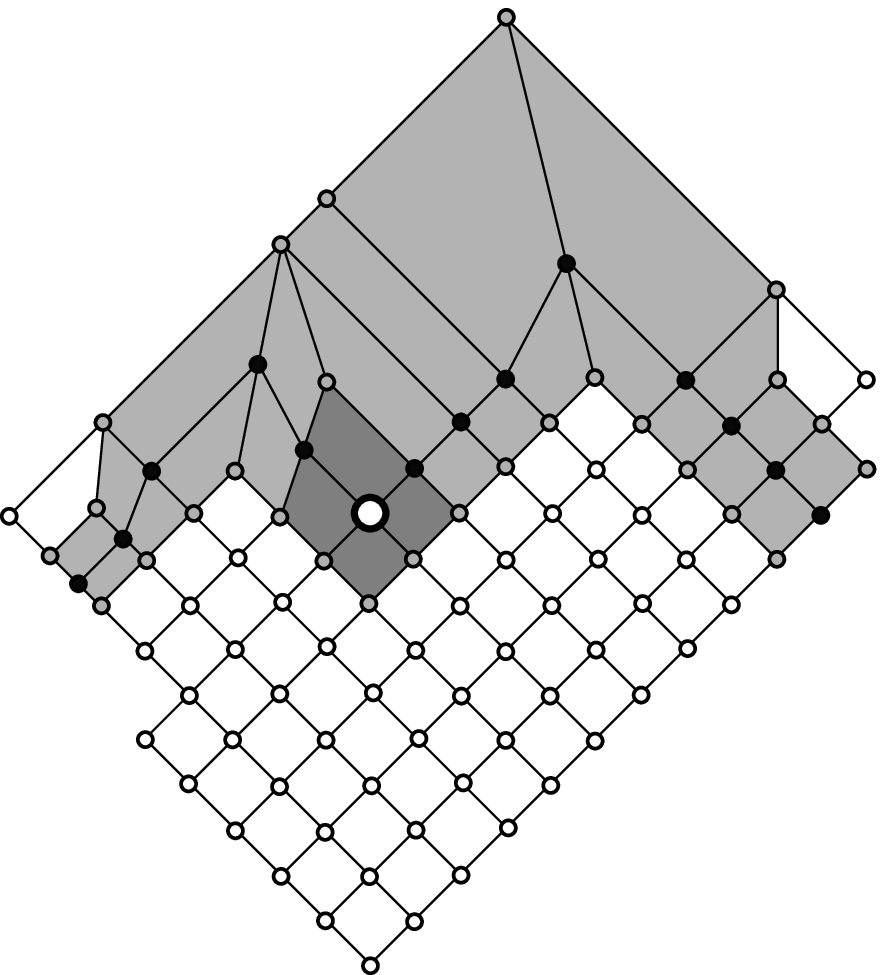}}
\caption{Resect this diagram at the element marked by the big circle
by deleting the black-filled elements}\label{fig:resect}

\bigskip

\centerline{\includegraphics[scale=1.0]{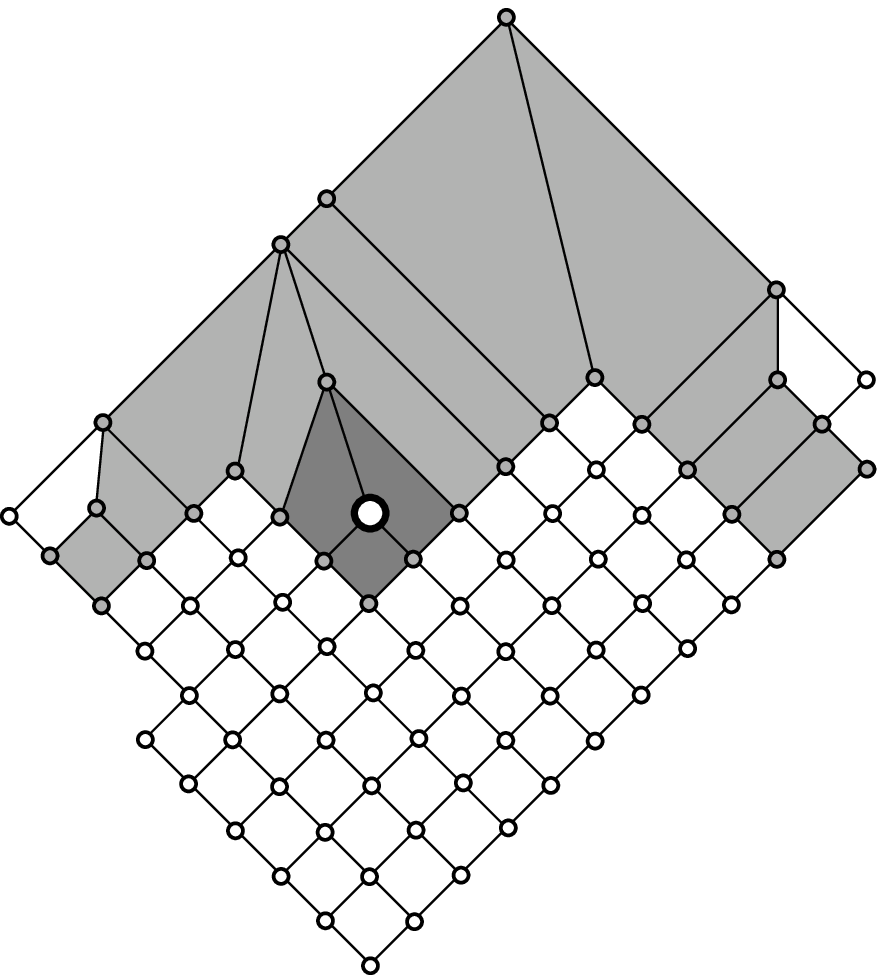}}
\caption{to obtain this diagram}\label{fig:afterresect}
\end{figure}

\begin{figure}
\centerline{\includegraphics[scale=1.0]{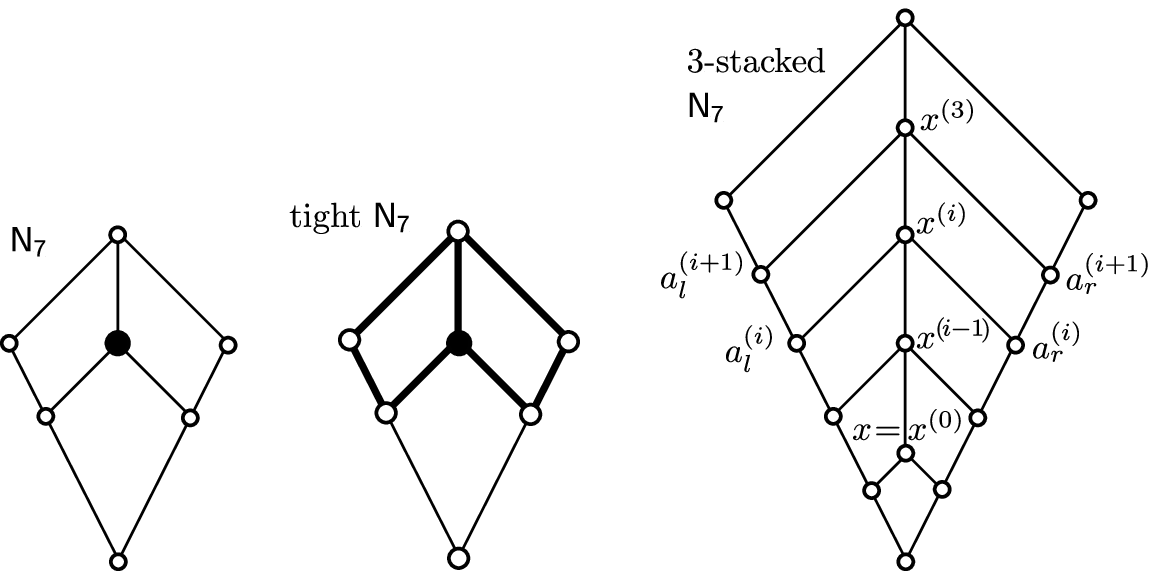}}
\caption{$\SN 7$ and its variants}\label{fig:Nseven}
\end{figure}

A lattice $L$ is \emph{slim} if it is finite 
and $\Ji L$ contains no three-element anti\-chain. 
Slim lattices are \emph{planar}, so we will consider \emph{planar diagrams} of slim semimodular lattices, \emph{slim semimodular diagrams}, for short. 

For the basic concepts and notation, 
we refer the reader to G.~Gr\"atzer \cite{LTF} 
and G. Cz\'edli and G.~Gr\"atzer \cite{CGa}.

\subsection*{Outline} Section~\ref{SectConstr} introduces resections. 
Section~\ref{SecResults} states the main result.  
Section~\ref{SecSchemes} recalls some known results 
on slim semimodular lattices 
and proves some facts on (the inverse of) resection schemes. 
Section~\ref{SecMainProof} contains the proof of the main result. 

\section{The construction}\label{SectConstr}
Let $D$ be a slim semimodular diagram. 
Two prime intervals of $D$ are \emph{consecutive}
if they are opposite sides of a $4$-cell (see Section~\ref{SecSchemes}). 
As in G.~Cz\'edli and E.\,T.\ Schmidt~\cite{CS11},
maximal sequences of  
consecutive prime intervals form a \emph{$\SC2$-trajectory}. 
So a $\SC 2$-trajectory is an equivalence class of the transitive reflexive closure of the ``consecutive'' relation. 
 
Similarly, let $A$ and $B$ be two cover-preserving $\SC 3$-chains of $D$. 
If they are opposite sides of a cover-preserving 
$\SC 3\times \SC 2$, then $A$ and $B$ are called \emph{consecutive}. 
An equivalence class of the transitive reflexive closure 
of this ``consecutive'' relation 
is called a \emph{$\SC 3$-trajectory}. 

We recall the basic properties of $\SC 2$-trajectories from \cite{CS11} and \cite{CSb}; 
they also hold for  $\SC 3$-trajectories.
For $i\in\set{2,3}$, a $\SC  i$-trajectory goes from left to right 
(unless otherwise stated); 
they do not branch out. 
A $\SC  i$-trajectory is of two types:
an \emph{up-trajectory,} which goes up (possibly, in zero steps) 
and a \emph{hat-trajectory,} which goes up (possibly in zero steps), 
then turns to the lower right, 
and finally it goes down (possibly, in zero steps). 

Note that the left and right ends of a $\SC 2$-trajectory
are on the boundary of~$L$; this may fail for a $\SC 3$-trajectory.

The \emph{elements} of a $\SC  i$-trajectory 
are the elements of the $\SC  i$-chains forming~it.
Let $A$ be a cover-preserving $\SC  i$-chain in $D$. 
By planarity, there is a unique $\SC  i$-tra\-jectory through $A$. 
The $\SC  i$-chains of this trajectory to the left of $A$ and including $A$
form the \emph{left wing of $A$}. 
The \emph{right wing of $A$} is defined analogously. 
  
Next, let $B$ be a cover-preserving 
$\SC 3^2=\SC 3 \times \SC 3$ of the diagram $D$. 
Let~$W_l$ be the left wing of the upper left boundary of $B$ and 
let $W_r$ be the right wing of the upper right boundary of $B$. 
Assume that $W_l$ and $W_r$ terminate on the boundary of $D$ 
(that is, the last $\SC 3$-chains are on the boundary of $D$). 
In this case, the collection of elements of $S = B  \uu  W_l  \uu  W_r$
is called a \emph{$\SC 3$-scheme} of $D$, see Figure~\ref{fig:resect}
for an example. 
The elements of $W_l$ and $W_r$ form the \emph{left wing} 
and the \emph{right wing} of this $\SC 3$-scheme, respectively, 
while $B$ is the \emph{base}. 
The middle element of $S$ is the \emph{anchor} of the scheme. 
A $\SC 3$-scheme is uniquely determined by its anchor. 
Of course, $D$ may have cover-preserving $\SC 3^2$'s  
that cannot be extended to $\SC 3$-schemes. 
For example, the slim semimodular diagrams 
in Figure~\ref{fig:examples} have cover-preserving 
$\SC 3^2$ sublattices but no $\SC 3$-schemes.

\begin{figure}
\centerline{\includegraphics[scale=1.0]{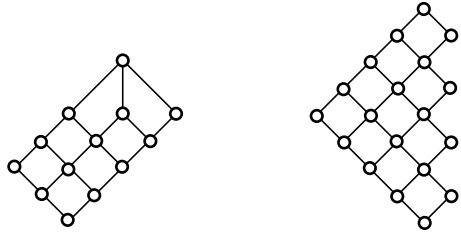}}
\caption{Some slim semimodular diagrams}\label{fig:examples}
\end{figure}

The concept of a $\SC 2$-scheme 
and the related terminology are analogous, 
see Figures~\ref{fig:afterresect} 
and \ref{fig:system} for two examples. 
The base of a $\SC 2$-scheme is a cover-preserving $\SN 7$, 
and its wings are in $\SC 2$-trajectories. 
The middle element of the  base is again called the anchor, 
and it determines the $\SC 2$-scheme. 
Since $\SC 2$-trajectories always reach the boundary of $D$, 
each cover-preserving $\SN 7$ sublattice 
is the base of a unique $\SC 2$-scheme.

For $i\in\set{2,3}$ and a $\SC  i$-scheme $S$, 
we define the \emph{upper boundary}, the \emph{lower boundary}, 
and the \emph{interior} of $S$ as expected.

Let $S$ be a $\SC 3$-scheme of a slim semimodular diagram $D$. 
By removing all the interior elements of $S$ but its anchor, 
we obtain a new slim semimodular diagram,~$D'$, 
and $S$ turns into a $\SC 2$-scheme of $D'$. 
We say that $D'$ is obtained from $D$ by a \emph{resection}; 
this is illustrated in Figures~\ref{fig:resect} 
and \ref{fig:afterresect}. 
The reverse procedure, transforming a $\SC 2$-scheme  
to a $\SC 3$-scheme by adding new interior elements, 
is called an \emph{insertion}.  

\section{The results}\label{SecResults}

Following D.~Kelly and I.~Rival~\cite{KR75}, 
we call two planar diagrams \emph{similar} 
if~there is a bijection $\gf$ between them such that 
$\gf$ preserves the left-right order of the upper covers 
and of the lower covers of an element. 
We~are interested in diagrams only up to similarity. 

A \emph{grid} is a planar diagram of the form 
$\SC m\times\SC n$ for $m,n\geq 2$.
We obtain a slim distributive diagram from a grid by a sequence of steps; 
each step omits a doubly irreducible element from a boundary chain. 
Our main result generalizes this to  
slim semimodular lattice diagrams.

\begin{theorem}\label{thmmain}
Slim semimodular lattice diagrams are characterized as diagrams 
obtained from slim distributive lattice diagrams by a sequence of resections.
\end{theorem}

The proof of this theorem now appears clear. 
Let $D$ be a slim semimodular lattice diagram. 
Find in it a covering $N_7$ as in Figure~\ref{fig:afterresect}. 
Perform an insertion to obtain the diagram of Figure~\ref{fig:resect}.
The diagram of Figure~\ref{fig:resect} has one fewer covering $N_7$.
Proceed this way until we obtain a diagram without covering $N_7$-s, 
that is, a covering $N_7$ diagram.

\begin{remark}\label{R:notwork}
The argument of the last paragraph does not necessarily work. 
Start with the first diagram in Figure~\ref{fig:endless}.
Apply an insertion at the black-filled element, 
to obtain the second diagram. 
Apply an insertion at the gray-filled element of the second diagram, 
to obtain the third diagram. And so on. 
It is clear that the number of covering $N_7$-s
is not diminishing.
\end{remark}

\begin{figure}
\centerline{\includegraphics[scale=1.0]{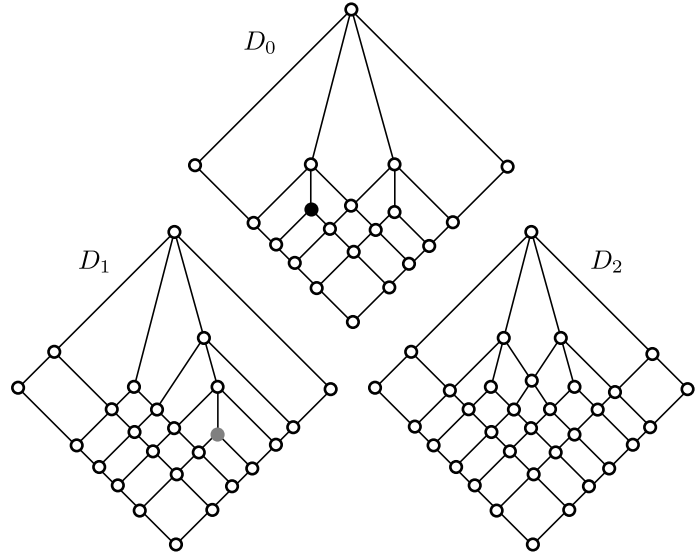}}
\caption{The process does not stop}\label{fig:endless}
\end{figure}

We define a \emph{weak corner} of a planar semimodular diagram $D$
as an element~$x$ on the boundary of $D$ with the properties:
\begin{enumeratei}
\item $x$ is doubly irreducible;
\item $x$ is not comparable to some $y \in D$.
\end{enumeratei} 
If $x$ is a weak corner such that its lower cover, $x_*$, 
has exactly two covers 
and its upper cover, $x^*$, has exactly two lower covers, 
then we call $x$ a \emph{corner}.
As~defined in G.~Gr\"atzer and E.~Knapp~\cite{GK09}, 
a planar diagram (and the corresponding lattice) is \emph{rectangular}, 
if it has exactly one left weak corner 
and exactly one right weak corner, 
and these two elements are complementary. 
Slim semimodular diagrams can be obtained from slim rectangular diagrams 
by~removing corners, one-by-one. 
Moreover, only slim semimodular diagrams can be obtained this way.  
So we get:

\begin{corollary} \label{rectcorRo}
Slim rectangular diagrams are characterized 
as diagrams obtained from grids by a sequence of resections.
\end{corollary} 

\section{Schemes}\label{SecSchemes}

Let $D$ be a slim semimodular diagram. 
By  G.~Gr\"atzer and E.~Knapp~\cite{GK07} and 
G.~Cz\'edli and E.\,T.~Schmidt~\cite[Lemma 2]{CSa}, 
an element of $D$ has at most two covers.
We also know from \cite[Lemma 6]{CSa} that
a join-irreducible element is on the boundary of $D$.
 
Let $a<b$ in a planar diagram $D$, 
and assume that $C_1$ and $C_2$ are maximal chains 
in the interval $[a,b]$ such that $C_1-\set{a,b}$ 
is strictly on the left of $C_2$,  $C_2-\set{a,b}$ 
is strictly on the right of $C_1$, and $C_1 \ii  C_2=\set{a,b}$. 
Then, following D.~Kelly and I.~Rival~\cite{KR75}, 
the intersection of the right of $C_1$ and the left of $C_2$ 
is called a \emph{region} of $D$. 
A region of $D$ is a planar subdiagram of $D$.  
Minimal regions are called \emph{cells}, 
and cells with four vertices (and four edges) are \emph{$4$-cells}. 
For a slim semimodular diagram $D$, 
\begin{equation}\label{covsqrefourcells}
\text{the $4$-cells and the covering squares of $D$ are the same.}
\end{equation}
Our proof relies heavily on the following two lemmas,
see G.~Gr\"atzer and E.~Knapp~\cite[Lemma 7]{GK07}, for the first and
G.~Gr\"atzer and E.~Knapp~\cite[Lemma 6]{GK07} and 
G.~Cz\'edli and E.\,T.~Schmidt~\cite[Lemma 15]{CSa}, for the second.

\begin{lemma}\label{lMawhadIssU}
Let $D$ be a planar lattice diagram. 
Then $D$ is slim and semimodular 
if{f} its cells are $4$-cells and no two distinct 
$4$-cells have the same bottom.
\end{lemma}

\begin{lemma}\label{lmAdisnonhet}
A slim semimodular diagram is distributive if{f} 
it has no cover-preserving $\SN 7$.
\end{lemma}

Let $\anchor iD$ denote the set of anchors of $\SC i$-schemes of $D$ 
for $i\in\set{2,3}$.
The set of \emph{interior elements} of $D$, 
that is, the set of those elements that are not on the boundary of $D$,  
is denoted by $\interior(D)$.   
Clearly,
\begin{equation}\label{sLczREQCv} 
\anchor 2D\subseteq \interior(D) \ii \Mi D\text.
\end{equation}

As in G.~Gr\"atzer and E.~Knapp~\cite{GK08}, 
an $\SN 7$ sublattice of $D$ is a \emph{tight $\SN 7$} 
if the thick edges in the middle diagram of 
Figure~\ref{fig:Nseven} represent coverings. 
A tight $\SN 7$ sublattice is always determined 
by its \emph{inner dual atom}, 
we call it the \emph{centre of $\SN 7$}, 
see the black-filled element lattice in the middle of Figure~\ref{fig:Nseven}.  

\begin{lemma}\label{tightnhetlma}
Let $D$ be a slim semimodular diagram
and let $u \in \interior(D)  \ii  \Mi D$. 
Then there exists a unique  tight $\SN 7$ sublattice of $D$ 
with $u$ as the anchor. 
Moreover, if  $[a_l,b_l]$, $[u,u^*]$, 
and $[a_r,b_r]$ are consecutive prime intervals of this sublattice, 
then this sublattice is $\set{u,u^*, a_l, b_l, a_r,b_r, a_l\wedge a_r}$. 
Conversely, the center of a tight $\SN 7$ sublattice always 
belongs to $\interior(D) \ii \Mi D$.
\end{lemma}

\begin{proof} Assume that $u\in \interior(D) \ii \Mi D$. 
Consider the $\SC 2$-trajectory $T$ containing $[u,u^*]$. 
Since $[u,u^*]$ is not on the boundary of $D$, 
this trajectory makes at least one step to the right, 
to a prime interval $[a_r,b_r]$. 
This step is a down-perspectivity since $u\in\Mi D$. 
Similarly, $T$ makes a down-perspective step to the left,  to $[a_l,b_l]$. 
By D.  Kelly and I. Rival \cite[Lemma 1.2]{KR75}, we obtain easily that $b_l\wedge b_r\leq u$. 
Thus $b_l\wedge b_r=b_l\wedge u\wedge b_r\wedge u=a_l\wedge a_r$, 
and we conclude that $\set{u,u^*, a_l, b_l, a_r,b_r, a_l\wedge a_r}$ 
is a tight $\SN 7$ sublattice of $D$.  

Observe that a tight $\SN 7$ sublattice with center 
$u$ is determined by those covering squares 
(that is, cover-preserving $\SC 2^2$ sublattices) of this sublattice 
that contain $[u,u^*]$ as an upper prime interval. 
But these covering squares are $4$-cells by \eqref{covsqrefourcells}, 
and $[u,u^*]$ is the upper edge of at most two $4$-cells.  
Hence, apart from left-right symmetry, 
there is only one tight $\SN 7$ sublattice with center $u$,
as described in the last paragraph. 
This proves the first two parts of the statement. 
The last part is trivial. 
\end{proof}

As in G.~Gr\"atzer and E.~Knapp~\cite{GK09}, 
a \emph{cover-preserving $m$-stacked $\SN 7$ $($sub\-lattice$)$} of $D$ 
is a cover-preserving sublattice isomorphic 
to the $(7+3m)$-element diagram given, 
for $m=3$, on the right of  Figure~\ref{fig:Nseven}.  
A cover-preserving \text{$0$-stacked} $\SN 7$ is a cover-preserving $\SN 7$. 

\begin{lemma}\label{mstackedlemma} Let $R$ be a cover-preserving $m$-stacked $\SN 7$ sublattice of $D$. Then $R$ is a region of $D$. Furthermore, if $R'$ a  cover-preserving $m$-stacked $\SN 7$ sublattice of $D$ such that $\interior(R) \ii \interior(R')\neq\varnothing$, then $R'=R$.
\end{lemma}

\begin{proof} Since $R$ consists of adjacent covering squares, which are $4$-cells by 
\eqref{covsqrefourcells}, it follows easily that $R$ is a region. 
Let $x_0\prec \cdots \prec x_m$ be the interior of~$R$. 
Assume that $t\in\interior(R)$. Then $t=x_j$, 
and this $j$ is recognized as follows: 
there is a sequence $t=t_0\succ \cdots \succ t_j$ such that 

(a) $t_j$ has only two lower covers; 

(b) the~$t_i$ have three lower covers for $i\in\set{1,\ldots,j-1}$;

(c) $t_i$ is the middle lower cover of $t_{i-1}$, for $i\in\set{1,\ldots,j}$. 

\noindent It follows that $t$ determines $\interior(R)$, 
which clearly determines the whole $R$ via adjacent $4$-cells. 
Hence if $t\in \interior(R) \ii \interior(R')$, then $R=R'$.
\end{proof}

In view of Lemma~\ref{mstackedlemma}, cover-preserving $m$-stacked $\SN 7$ sublattices of $D$ are  also called \emph{$m$-stacked $\SN 7$ regions}.
%Let $\tup{LC}_3(D)$ denote the set of all elements in $D$ 
%that have exactly three lower covers. 
For a meet-irreducible element $x \in D$ in the interior of $D$
define $x(0)=x$. If the meet-irreducible element $x(i)$ is already defined
and $x(i)^*$ 
\begin{enumeratea}
\item is meet-irreducible,
\item is in the interior of $D$,
\item covers exactly three elements,
\end{enumeratea}
then define $x(i + 1) = x(i)^*$. 
The \emph{rank} of $x$, $\rank_{D}(x)$,  
is the largest $m$ such that $x(m)$ is defined. 
By \eqref{sLczREQCv},  each $x\in\anchor 2D$ has a rank. 
For another description of $\rank_{D}(x)$, 
where $x\in\anchor 2D$, see Corollary~\ref{eZaranG}.

\begin{lemma}\label{sTcklmMa} Let $D$ be a slim semimodular diagram.  
Let $x \in \anchor2D$ and $\rank_{D}(x) = m$. 
Then the following statements hold:
\begin{enumeratei}
\item\label{sTcklmMaa} The element $x$ has exactly two lower covers. 
\item\label{sTcklmMab} For $i\in\set{0,\ldots, m}$, 
there exists a unique $i$-stacked $\SN 7$ region $R_i$ of $D$ 
such that $\interior(R_i) = \set{x = x(0) \prec \cdots \prec x(i)}$.
\item\label{sTcklmMac} The interior of the $\SC 2$-scheme anchored 
by $x$ is $\interior(R_m)$.
\end{enumeratei}
\end{lemma}

\begin{proof} \eqref{sTcklmMaa} is trivial.

To prove \eqref{sTcklmMab}, let $H(i)$ denote the condition 
``there exists a unique $i$-stacked $\SN 7$ region $R_i$ of $D$ such that 
$\interior(R_i)=\set{x=x(0)\prec\cdots\prec x(i)}$''. 
Observe that $x$ is the center of a cover-preserving $\SN 7$ sublattice $R_0$ 
by definition. It is a $0$-stacked $\SN 7$ region. 
Since $R_0$ is also a tight $\SN 7$ sublattice, 
$R_0$ is uniquely determined by Lemma~\ref{tightnhetlma}. 
This proves $H(0)$. 

Next, let $1\leq i\leq m$ and assume that $H(i-1)$ holds. 
Since $x(i-1)$, the anchor of $R_{i-1}$, has only one cover, 
it follows that
$x(i)$ is the top of $R_{i-1}$; 
for an illustration, see the diagram 
on the right of Figure~\ref{fig:Nseven} with $i=2$. 
Since $x(i)$ is defined, 
it satisfies (a)--(c). 
Hence the lower covers of $x(i)$ in $D$ are exactly the same 
as the dual atoms of $R_{i-1}$, namely, 
the left dual atom $a(i)_l$, the anchor $x(i-1)$, 
and the right dual atom $a(i)_r$ of $R_{i-1}$. 
Since $x(i)\in\interior(D) \ii  \Mi D$, 
the right wing starting from 
$[x(i), x(i)^*]$ has to make its first step downwards 
to $[a(i)_r,a(i+1)_r]$, 
where $a(i+1)_r$ is a uniquely determined element of $D$
because 
\[
   \set{a(i)_r,x(i), a(i+1)_r, x(i)\vee a(i+1)_r}
\] 
is a $4$-cell of $D$. 
By left-right symmetry, we also obtain a unique $4$-cell 
\[
   \set{a(i)_l,x(i), a(i+1)_l, x(i)\vee a(i+1)_l}.
\] 
Since 
\[\set{a(i)_l,a(i+1)_l,x(i), a(i)_r,a(i+1)_r}\] 
generates a (unique) tight $\SN 7$ sublattice by Lemma~\ref{tightnhetlma}, 
it follows that
\[
   a(i+1)_l\wedge a(i+1)_r=a(i)_l\wedge a(i)_r.
\]
This together with the fact that $R_{i-1}$ 
is a cover-preserving $(i-1)$-stacked $\SN 7$ sublattice 
implies that 
\[
   R_i=R_{i-1} \uu \set{a(i+1)_l,x(i)^*,a(i+1)_r}
\] 
is a cover-preserving $i$-stacked $\SN 7$ region 
with interior $\set{x=x(0)\prec \cdots\prec x(i)}$. 
The uniqueness of this region follows from Lemma~\ref{mstackedlemma}. 
Hence $H(i)$ holds for all $i\in\set{0,\ldots,m}$, 
proving part \eqref{sTcklmMab} of the lemma. 

Finally, \eqref{sTcklmMac} is obvious.
\end{proof}

\begin{corollary}\label{eZaranG} Let $D$ be a slim semimodular diagram, 
and let $x\in\anchor2D$. 
Then $\rank_{D}(x)$ is the largest number in the set 
\[\setm{k}{x \text{ is the middle atom of a $k$-stacked $\SN 7$-region}}.\]
\end{corollary}

\section{The proof of the main result}\label{SecMainProof}

We start with a simple consequence of Lemma~\ref{lMawhadIssU}:

\begin{lemma}\label{insreslmM}\hfill
\begin{enumeratei}
\item\label{insreslmMa} Let $D$ be a slim semimodular diagram 
and let $D'$ be obtained from $D$ by a resection at $u\in \anchor3D$. 
Then $D'$ is also a slim semimodular diagram, $u\in\anchor2{D'}$, 
and up to similarity, $D$ is obtained from $D'$ by an insertion at $u$.
\item\label{insreslmMb} Conversely, let $D$ be a slim semimodular diagram 
and let $D'$ be obtained from $D$ by an insertion at $u\in \anchor2D$. 
Then $D'$ is also a slim semimodular diagram, $u\in\anchor3{D'}$, 
and up to similarity, $D$ is obtained from $D'$ by a resection at $u$.
\end{enumeratei} 
\end{lemma}

The following lemma is the major step in the proof of Theorem~\ref{thmmain}. 
Note that the inclusions in it are actually equalities, 
but we do not need---and do not prove---this.

\begin{lemma}\label{mainlemma} 
Let $D$ be a slim semimodular diagram and assume that $u\in\anchor2D$.  
Let $D'$ denote the diagram obtained from $D$ by performing an insertion at $u$. 

If $\rank_{D}(u)=0$, 
then \[\anchor 2{D'}\subseteq \anchor2D-\set{u}.\] 

If $\rank_{D}(u)>0$, 
then \[\anchor 2{D'}\subseteq (\anchor 2D-\set u) \uu \set{u^*},\]  
and $\rank_{D'}(u^*)=\rank_{D}(u)-1$.
\end{lemma}

\begin{figure}
\centerline{\includegraphics[scale=1.0]{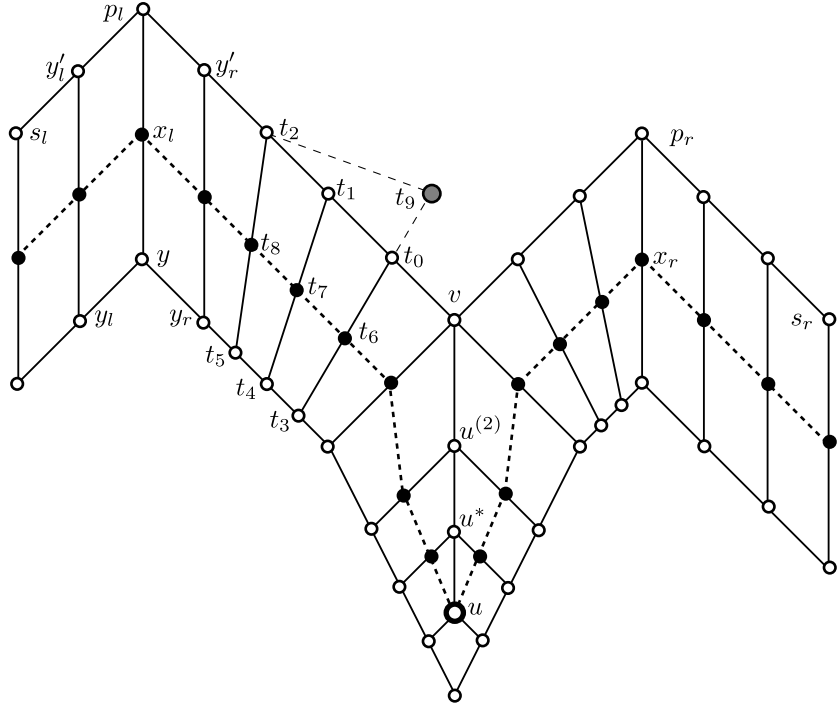}}
\caption{Insertion at $u$ ($t_9$ plays a role only in Case 3)}  \label{fig:system}
\end{figure}

\begin{proof} Let $S$ denote the $\SC 2$-scheme anchored by $u$.  
Let $I$ be the order-ideal of~$D$ generated by the lower boundary of $S$ and 
let $F$ be order filter generated by the upper boundary of $S$. 
Since $I \ii  S$ is the lower boundary of $S$ 
and $F \ii  S$ is the upper boundary  of $S$,
planarity implies that, for all $x_1,x_2\in D$,  
\begin{equation}\label{wizdzsconW}
\text{if $x_1\prec x_2$, $x_2\in F-S$, 
      and $x_1\notin F-S$, then $x_1\in F \ii  S$.}
\end{equation} 
By Lemma~\ref{sTcklmMa}, $u$ is the inner atom 
of a unique $\rank_{D}(u)$-stacked $\SN 7$ region, 
whose top we denote by $v$, see Figure~\ref{fig:system}. 
(Note that we utilize that $t_9$ exists and is placed in the diagram
as shown only in Case 3; in general, $t_9$ is not in $S$, and it may not be in $D$.)  
Let $p_l$ and $p_r$ denote the top elements of the wings. 
Let $s_l$ and $s_r$ denote the largest elements 
of the wings on the boundary of $D$. 
It is possible that  $s_l = p_l$, or $p_l = v$, or $s_r = p_r$, or $p_r= v$.

If $x_l$ in Figure~\ref{fig:system} 
(the element of $D'-D$ covered by $p_l$) 
belongs to  $\interior(D')$, then it has at least three lower covers in $D'$;
similarly for $x_r$. All the other elements of~$D'- D$ are meet-reducible. 
Thus $\anchor2{D'}\subseteq D$. 
So we have to show that every element $w$ of  
$D \ii  \anchor2{D'} = \anchor2{D'}$ is in $\anchor 2D$. 

Since $D$ can be partitioned into
\[
   I \uu (F-S) \uu (F \ii S) \uu  (D - (I \uu F)),
\]
the condition $w \in D$ splits into four cases as to which block in this partition $w$ belongs to.

\emph{Case 1: $w\in I$.}  
If $w\notin S$, then $w^* \in I\subseteq D$ by the dual of \eqref{wizdzsconW}, 
the unique cover-preserving $\SN 7$ sublattice 
is in $I\subseteq D$, and $w\in\anchor2D$, as required by the lemma. 
Therefore, we can assume that $w$ belongs to the lower boundary of $S$. 
Since $w\in\interior(D') \ii  {\Mi{D'}}$, 
it has to be where a wing (properly) turns down, $w = y$ in Figure~\ref{fig:system}
(or symmetrically, on the right).  
It~has exactly two lower covers by Lemma~\ref{sTcklmMa}.  
Thus these lower covers, $y_l$ and $y_r$ in Figure~\ref{fig:system}, 
also belong to the lower boundary of $S$.  
We use the notation $y_l'$ and $y_r'$ as in Figure~\ref{fig:system}. 
Lemma~\ref{tightnhetlma} yields that 
$\set{y,p_l, y_l, y_l', y_r,y_r', y_l\wedge y_r}$ 
is a tight $\SN 7$ sublattice of $D$. 
Since $y\in\anchor2{D'}$ yields that 
$y_l\wedge y_r\prec y_l$ and $y_l\wedge y_r\prec y_r$, 
this tight $\SN 7$ sublattice is a cover-preserving $\SN 7$ sublattice. 
Hence $y \in \anchor2D$, as required. 

\emph{Case 2: $w\in F-S$.}  
The element $w$ has exactly two lower covers, $w_l$ and $w_r$, 
by Lemma~\ref{sTcklmMa}. 
They belong to $F$, and we have that $w_l\wedge w_r\prec w_l$ 
and $w_l\wedge w_r\prec w_r$. 
If at least one of $w_l$ and $w_r$ does not belong to~$S$ 
(equivalently, to its the upper boundary, $F \ii  S$), 
then $w_l\wedge w_r\in F$ by \eqref{wizdzsconW}, 
whence the  cover-preserving $\SN 7$ sublattice 
determined by $w$ belongs to $F$ and $w\in\anchor2D$, as required. 
Hence we can assume that $w_l$ and $w_r$ are 
on the upper boundary of $S$ but $w_l\wedge w_r\notin F$. 
Since $w_l\parallel w_r$, the only possibility, up to left-right symmetry, 
is that $w=w_l\vee w_r$ equals $p_r$. 
However, this case is excluded by Lemma~\ref{sTcklmMa} 
since $p_r$ has at least three lower covers.  

\emph{Case 3: $w \in F \ii S$.}  
Let $w = t_1$ be on the upper boundary of $S$, as in Figure~\ref{fig:system}. 
Since it has only two lower covers by Lemma~\ref{sTcklmMa} 
and belongs to $\interior(D')$, 
we conclude that $t_1\notin\set{v,p_l,p_r, s_l, s_r}$. 
Hence there are elements $t_0,t_2$ in the upper boundary of $S$, 
in the same wing as $t_1$, such that $t_0\prec t_1\prec t_2$. 
We~use the notation $t_4,\ldots, t_8$ as in Figure~\ref{fig:system}. 
Consider the cover-preserving $\SN 7$ sublattice in $D'$ 
that is anchored by $t_1$. 
Since this sublattice is also a tight $\SN 7$ sublattice in $D'$, 
it is  $\set{t_0,t_1,t_2,t_6,t_7,t_8, t_9}$ 
with rightmost dual atom $t_9$ by Lemma~\ref{tightnhetlma}.  
Therefore, applying Lemma~\ref{tightnhetlma} again, 
the tight $\SN 7$ sublattice determined by  
$t_1$ in $D$ is  $\set{t_0,\dots, t_5, t_9}$. 
It is a cover-preserving  $\SN 7$ sublattice 
since $t_3\prec t_6$ and $t_3\prec t_4$. 
Thus $t_1\in \anchor2D$, as required. 

\emph{Case 4: $w\in D - (I \uu F)$.}
Notice that $w\in \interior(S)$. 
By Lemma~\ref{sTcklmMa}, $w$ belongs 
to the interior of the unique $m$-stacked $\SN 7$ region $R_m$
with centre $u$, where $m=\rank_{D}(u)$.
Assume first that $m=0$. Then $w=u$, whence it does not belong to 
$\anchor2{D'}$ since it has two upper covers in~$D'$. 
Secondly, assume that $m>0$. 
Then, clearly again, $u\notin \anchor2{D'}$. 
Moreover, of the other elements in the interior of $S$, 
that is, in \[\interior({S})=\interior({R_m})=\setm{u^{(i)}}{1\leq i\leq m},\]   
the element $u^*= u^{(1)}$ is the only one in $\anchor2{D'}$. \qedhere
\end{proof}

\begin{proof}[Proof of Theorem~\ref{thmmain}] 
Let $D$ be a diagram.
If it is obtained from a slim distributive diagram 
by a sequence of restrictions, 
then $D$ is a slim semimodular diagram by Lemma~\ref{insreslmM}. 
Conversely, assume that $D$ is a slim semimodular diagram. 
By virtue of Lemma~\ref{insreslmM}, 
it suffices to show that we can obtain 
a slim distributive diagram from $D$ by a finite sequence of insertions. 
That is, we want a finite sequence $D=D_0, D_1,\ldots$ of diagrams 
such that $D_{i+1}$ is obtained from $D_i$ by an insertion, 
and the last member of the sequence is distributive. 
If~$D_i$ is distributive, then it is the last member of the sequence, 
and we are ready. If it is non-distributive, 
then $\anchor2 {D_i}$ is non-empty by Lemma~\ref{lmAdisnonhet}. 
Pick an element $u_i\in \anchor2{D_i}$ such that $\rank_{D_i}(u_i)$ 
is the smallest member of $\setm{\rank_{D_i}(x)}{x\in \anchor2 {D_i}}$, 
and perform an insertion at $u_i$ to obtain $D_{i+1}$. 
This procedure terminates in finitely many steps by Lemma~\ref{mainlemma}.
\end{proof}

\begin{proof}[Proof of Corollary~\ref{rectcorRo}] 
If we perform an insertion to obtain $D'$ from $D$, 
then the weak corners of $D'$ are the same as those of $D$, $0_{D'}=0_D$, 
and $1_{D'}=1_D$. 
Hence our statement follows from G.~Gr\"atzer and E.~Knapp~\cite[Lemma 6]{GK09} and the argument used in the proof of  Theorem~\ref{thmmain}.
\sajat{\cite[Lemma 6]{GK09}:Let $L$ be a planar semimodular lattice with exactly one left corner and
exactly one right corner. Then $L$ is a rectangular lattice iff 1 is join-reducible and 0 is meet-reducible.}
\end{proof}

There are some efficient ways to check whether a planar diagram 
is a slim semimodular lattice diagram; 
in addition to Lemma~\ref{lMawhadIssU}, see \cite[Theorems 11 and 12]{CSa}. 
The following test follows trivially from the proof of Theorem~\ref{thmmain}. 

\begin{lemma}\label{lemHtoCZs} 
Let $D$ be a planar diagram. 
Construct the sequence 
\[
   D=D_0, D_1,\ldots
\] 
as in the proof of Theorem~\ref{thmmain}. 
Then $D$ is a slim semimodular lattice if{f} 
the sequence terminates with a planar distributive lattice. 
\end{lemma}

Remark~\ref{R:notwork} points out that the clause
``as in the proof of Theorem~\ref{thmmain}'' in Lemma~\ref{lemHtoCZs}
cannot be dropped.

\end{document}